\documentclass[11pt,oneside]{amsart}
\usepackage{amsmath}
\usepackage{amsfonts}
\usepackage{amssymb}
\usepackage{graphicx}
\usepackage[mathscr]{eucal}

\theoremstyle{plain}

\newtheorem*{conjecture}{Conjecture}
\newtheorem*{corollary}{Corollary}

\newtheorem{lemma}{Lemma}

\newtheorem{proposition}{Proposition}
\theoremstyle{remark}
\newtheorem{remark}{Remark}

\usepackage[height=215mm,width=155mm]{geometry}
\begin{document}
\title[On face numbers of neighborly cubical polytopes]{On face numbers of \\neighborly cubical polytopes}
\author{L\'aszl\'o Major}
\address{L\'aszl\'o Major\newline%
\indent Demartment of Statistics,   \newline%
\indent Faculty of Social Sciences, \newline%
\indent  E\"{o}tv\"{o}s Lor\'and University, Budapest \newline%
\indent  P\'azm\'any P\'eter s\'eét\'any 1/A, H-1117, Budapest, Hungary }
\email{major@tatk.elte.hu}%
\date{Jan 13, 2015}
\subjclass[2010]{Primary 05A10, 52B11; Secondary 52B12, 52B05}
\keywords{neighborly cubical polytope, cubical polytope, Pascal's triangle,  face vector, short cubical $h$-vector}

\begin{abstract}
Neighborly cubical polytopes are known as the cubical analogues of the cyclic  polytopes. Using the  short cubical $h$-vectors  of  cubical polytopes (introduced  by Adin), we derive  an explicit formula for the face numbers of the neighborly cubical polytopes. These face numbers form a unimodal sequence.
\end{abstract}%
\maketitle
\vspace{4mm}

The vector $\textbf{a}=(a_0,\ldots, a_{d-1})$ is  called \textit{unimodal} if for some (not necessarily unique) index $i$, $(a_0,\ldots, a_i)$ is non-decreasing  and  $(a_i,\ldots, a_{d-1})$ is non-increasing. If that is the case, we say that the unimodal vector $\textbf{a}$ \textit{peaks} at $i$. The question of unimodality of the members of certain classes of vectors  has been of  long-standing interest   in algebra, combinatorics and geometry  (see e.g. \cite{B1,sta1}).   The proof of unimodality of $\textbf{a}=(a_0,\ldots,a_{d-1})$, when all $a_i>0$, is sometimes (as in the case of Proposition \ref{thu}) based on proving the stronger property of \textit{log-concavity} ($a_{i-1}a_{i+1}\leq a_{i}^2$ for all $1<i<n$).

 By  \textit{$f$-vector} (or \textit{face vector}) we mean the vector $(f_0,\ldots,f_{d-1})$, where  $f_i$ is the number of $i$-dimensional proper faces of a $d$-polytope. The unimodality of face vectors of certain classes of polytopes is also extensively studied (see e.g. \cite{bj1,eck,maj,SZ1,zi1}). 
 
 The \textit{$d$-cube} (denoted by $C^d$)  is a polytope combinatorially equivalent to 
 the unit cube $[0, 1]^d$. Let $P$ be a $d$-polytope and let $0\leq r\leq d$. The \textit{$r$-skeleton} of  $P$ is the union of the $r$-dimensional faces of $P$. A $d$-polytope is called \textit{cubical} provided all its facets are combinatorially equivalent to the $(d-1)$-cube. A cubical $d$-polytope is called \textit{neighborly cubical} provided its $(\lfloor \frac{d}{2} \rfloor -1)$-skeleton is combinatorially equivalent to the $(\lfloor \frac{d}{2} \rfloor -1)$-skeleton of a cube. If this cube is the $n$-cube ($n\geq d$), then the neighborly cubical polytope in question has $2^n$ vertices.
 
  The concept of neighborly cubical polytopes was introduced by Babson, Billera and Chan \cite{bil}.   Joswig and Ziegler proved in \cite{jos} that there exists a $d$-dimensional neighborly cubical polytope  with $2^n$ vertices for any $n\geq d\geq 2$. They  constructed neighborly cubical polytopes as linear projections of cubes. 
 
   The $f$-vector of a neighborly cubical polytope is determined by the parameters $n$ and $d$. According to the definition, the first half of the $f$-vector of a  neighborly cubical $d$-polytope (with $2^n$ vertices) is identical to the first half of the $f$-vector of the $n$-cube. Formally, if $P$ is  a  neighborly cubical  $d$-polytope (with $2^n$ vertices), then 
 \begin{equation}\label{kocka}
f_i(P)=f_i(C^n)=2^{n-i}\binom{n}{i} \hspace{4mm}  \text{for} \hspace{3mm} 0\leq i \leq \Big\lfloor \frac{d}{2} \Big\rfloor -1
 \end{equation}

The following important combinatorial invariant of a cubical polytope is introduced  by Adin \cite{adi}. 
Let $P$ be a cubical $d$-polytope with f-vector $\textbf{f}=(f_0,\ldots,f_{d-1})$ and let $H$ be the $d\times d$ matrix given by
\begin{equation*}\label{fhH}
H(i,j)=2^{-j}\binom{d-i-1}{d-j-1}, \hspace{4mm}  \text{for} \hspace{3mm} 0\leq i,j \leq d-1
\end{equation*}
  Define the \textit{short cubical h-vector} $\textbf{h}^{(sc)}=(h_0^{(sc)},\ldots,h_{d-1}^{(sc)})$ of $P$ by $\textbf{h}^{(sc)}=\textbf{f}\cdot H^{-1}$.
Equivalently, the face vector of $P$ can be expressed by
\begin{equation}\label{fh}
\textbf{f}=\textbf{h}^{(sc)}\cdot H
\end{equation}
 In the sequel, we call the vector $\textbf{h}=(h_0,\ldots,h_{d-1})=2^{-d}\cdot \textbf{h}^{(sc)}$  \textit{scaled short cubical $h$-vector} (or \textit{scaled $h$-vector} for short) of $P$.
The scaled short cubical $h$-vector of a neighborly cubical polytope
is a symmetric vector of integers (see Lemma \ref{hh} and Proposition \ref{h3}), just like the short cubical $h$-vector $(\textbf{h}^{(sc)})$ of any cubical polytope.

For example, each component of the scaled $h$-vector of a cube is $1$ and the scaled $h$-vector of a $4$-dimensional neighborly cubical polytope with $64$ vertices is $(4,12,12,4)$.

Due to  (\ref{fh}), to know the scaled $h$-vector of a neighborly cubical polytope is equivalent to knowing its $f$-vector. The relations between the scaled $h$-vector and  the $f$-vector  can be visualized following an observation of Stanley \cite{sta} as adapted
by Lee in \cite{lee}. We replace the numbers $\binom ii$ of Pascal's triangle by the components $h_i$ for $0\leq i \leq d-1$ and the numbers $\binom i0$ by the numbers $2^{n-d+i}$  for $0\leq i \leq d$, then we compute the internal entries  as the upper left neighbour plus two times the upper right neighbour. The $f$-vector of the neighborly cubical polytope emerges in the $(d)^{th}$ row of this modified Pascal's triangle.\vspace{2.7mm} \\
\textbf{Example.} We compute the $f$-vector of a $4$-dimensional neighborly cubical polytope with $64$ vertices by using the above method.

\setlength{\unitlength}{0.75mm}
\begin{picture}(100,68)(0,0.5)

\put(90,58){${4}$} 

\put(85,50){$8$}\put(95,50){${12}$} 

\put(80,42){$16$}\put(90,42){$32$} \put(100,42){${12}$}

\put(75,34){$32$} \put(85,34){$80$}\put(94,34){$56$} \put(104.3,34){${4}$}

\put(60,26){$f$ =}\put(68,26){ $( 64$} \put(80,26){$192$} \put(90,26){$192$} \put(100,26){$64)$ }

\put(60,30.9){\line(1,0){49.2}}

\end{picture}
\vspace{-21mm}\\
\begin{lemma}[Adin (47) in \cite{adi}]\label{hh} Let $P$ be a neighborly cubical $d$-polytope, then for the scaled $h$-vector of $P$ we have:
\begin{enumerate}
\item[(\textit{i})] $h_i=h_{d-i-1}$ for $0\leq i \vspace{1.5mm}\leq d-1$,
\item[(\textit{ii})]$\displaystyle\sum_{i=0}^{d-1}(-1)^ih_i=1$ \textrm{ if $d$ is odd}.
\end{enumerate}
\end{lemma}
\begin{lemma}\label{hh1}Let  $P$ be a neighborly cubical $d$-polytope with $2^n$ $(n\geq d\geq 2)$ vertices, then
$$h_i(P)=f_i(C^{n-d+i})\hspace{4mm}\text{for all}\hspace{2mm}0 \leq i\leq \Big\lfloor\frac d2 \Big\rfloor-1,$$
where $C^{n-d+i}$ denotes the $(n-d+i)$-cube.
\end{lemma}
\begin{proof}According to (\ref{kocka}) and the definition of the   scaled $h$-vector, we have to show that for all $0\leq i\leq \lfloor\frac d2 \rfloor-1$,
$$\sum_{k=0}^i(-1)^{i-k}2^{k-d}\binom{d-k-1}{d-i-1}f_k(C^n)=f_i(C^{n-d+i}).$$
The above equation is equivalent to the following one: 
\begin{equation}\label{jaj}
\sum_{k=0}^i(-1)^{i-k}\binom{d-k-1}{d-i-1}\binom nk= \binom{n-d+i}{i}.
\end{equation}
Equation (\ref{jaj}) can be proved by induction (for a proof of an analogous combinatorial identity see e.g.  \cite[Theorem 9.2.2]{gru}).
\end{proof}
The following proposition  gives  the components of the scaled $h$-vector of a neighborly cubical polytope.  
\begin{proposition}\label{h3} The scaled $h$-vector $(h_0,\ldots,h_{d-1})$ of a $d$-dimensional  neighborly cubical polytope with $2^n$ $(n\geq d\geq 2)$ vertices is given \vspace{2mm}by
 \[ h_i=\left\{ \begin{array}{ll}
2^{n-d} \displaystyle\binom{n-d+i}{i} & \textrm{for \;$i\leq \lfloor\frac d2 \rfloor-1$} \\
\displaystyle\sum_{j=0}^{n-d}2^j\displaystyle\binom{\frac{d-3}{2}+j}{j} &
\textrm{for \;$i= \frac {d-1}{2} $ if $d$ is odd} \\
2^{n-d} \displaystyle\binom{n-i-1}{d-i-1} &  \textrm{for \;$i> \lfloor\frac d2 \rfloor-1$} 
\end{array} \right. \]
\end{proposition}
\vspace{0.7mm}
\begin{proof}
We have to verify the statement concerning the  middle component of $\textbf{h}$  assuming that $d$ is odd (other cases  follow directly from Lemma \ref{hh} and Lemma \ref{hh1}). We prove by induction on the integer $n$.
If $n=d$, then the statement is obvious, because we have the $d$-cube in question. Otherwise, by using Lemma \ref{hh}, Pascal's rule $\binom{m}{k}=\binom{m-1}{k-1}+\binom{m}{k-1}$ and the induction hypothesis, we establish the statement for any integer $n\geq d$.
\end{proof}
\begin{remark}One can observe the following relation between the simplicial $h$-vector $\textbf{h}^{(s)}$ of a  neighborly simplicial $d$-polytope $S$ and the scaled $h$-vector $\textbf{h}$ of a neighborly cubical $(d+1)$-polytope $C$ (assuming that $d$ is odd):
$$\textbf{h}^{(s)}(S)=2^{n-d-1}\textbf{h}(C)$$
\end{remark} 
\begin{lemma}\label{le2}
The scaled $h$-vector of a  neighborly cubical polytope is log-concave. 
\end{lemma}
\begin{proof}
Let $P$ be  a $d$-dimensional neighborly cubical polytope with $2^n$ vertices and with the scaled $h$-vector  $(h_{0},\ldots,h_{d-1})$.
First, we assume that $d$ is even. Let $m=d/2-1$. The log-concavity of $(h_0,\ldots,h_{m})$ (and also the log-concavity of $(h_{m+1},\ldots,h_{d-1})$ because of the symmetry) follows from Theorem 1 of Su and Wang \cite{wang}. From  Proposition \ref{h3}, it is clear that $(h_0,\ldots,h_{m})$ is increasing and   $(h_{m+1},\ldots,h_{d-1})$  is decreasing, therefore the whole scaled $h$-vector of $P$ is log-concave. 

Now we assume that $d$ is odd.  Let us denote $\lfloor\frac d2  \rfloor-1$ by $m$ and $2^{n-d}\binom{n-d+m+1}{m+1}$ by $h$. Due to the above reasoning, we obtain that $(h_0,\ldots,h_{m})$ and $(h_{0},\ldots,h_{m},h)$ are both log-concave. Using the identity $\binom{l+1}{k+1}=\binom{l}{k}+\binom{l-1}{k}+\cdots + \binom{k}{k}$, one can show that $h_m<h_{m+1}< h$, therefore the whole scaled $h$-vector of $P$ is log-concave. 
\end{proof}

 \vspace{0mm}
\begin{proposition}\label{thu} The $f$-vector of a  neighborly cubical polytope is log-concave, hence unimodal.
\end{proposition} 
\begin{proof}
 Corollary 8.3 in \cite{B1} states that the positive vector $\textbf{b}=(b_0,\ldots, b_{d-1})$, where $$b_{k}=\sum_{i=0}^k \binom{d-i-1}{d-k-1}a_i, \hspace{2mm}\text{for } \hspace{2mm}0\leq k \leq  d-1,$$ is log-concave if $\textbf{a}=(a_0,\ldots, a_{d-1})$ is a positive log-concave vector. Since the componentwise product (Schur product) of two log-concave vectors is log-concave as well, Corollary 8.3  implies that the positive vector $\textbf{c}=(c_0,\ldots, c_{d-1})$, where $$c_{k}=2^{d-k}\sum_{i=0}^k \binom{d-i-1}{d-k-1}a_i, \hspace{2mm}\text{for } \hspace{2mm}0\leq k \leq  d-1,$$ is log-concave if $\textbf{a}=(a_0,\ldots, a_{d-1})$ is a positive log-concave vector.  Due to (\ref{fh}),  the $f$-vector of a  neighborly cubical polytope can be given by  
 \begin{equation*}f_{k}=2^{d-k}\sum_{i=0}^k \binom{d-i-1}{d-k-1}h_i, \hspace{2mm}\text{for } \hspace{2mm}0\leq k \leq  d-1.\end{equation*}
By Lemma  \ref{le2},  the scaled $h$-vectors of the neighborly cubical polytopes are always log-concave. Consequently, we obtain that the $f$-vectors of the neighborly cubical polytopes are log-concave indeed. 
\end{proof}

Joswig and Ziegler (Corollary 19. in \cite{jos}) gave a formula for the number of facets of  neighborly cubical polytopes.
 The following proposition gives the complete $f$-vector of  neighborly cubical polytopes.  This proposition can be proved by using the concept of the scaled $h$-vector and its properties given in Proposition \ref{h3}. The proof  is completely analogous to the proof of the corresponding result of neighborly simplicial polytopes (see e.g. \cite{zie}, Section \vspace{2mm}8.4).

\begin{proposition}\label{pfk} The $f$-vector $(f_0,\ldots,f_{d-1})$ of a $d$-dimensional  neighborly cubical polytope with $2^n$ $(n\geq d\geq 2)$ vertices is given \vspace{2mm}by
\begin{displaymath}
 f_k=\left\{ \begin{array}{ll}
2^{n-k}\displaystyle\sum_{i=0}^{\frac{d-2}{2}} 
\textstyle
\Big(\binom{d-i-1}{k-i}+\binom{i}{k-d+i+1}\Big)\binom{n-d+i}{i} &
\textrm{if $d$ is even} \\
2^{n-k}\Bigg(\displaystyle\sum_{i=0}^{\frac{d-3}{2}} 
\textstyle
\Big(\binom{d-i-1}{k-i}+\binom{i}{k-d+i+1}\Big)\binom{n-d+i}{i}+\displaystyle\sum_{j=0}^{n-d}2^{-j}\textstyle\binom{\frac{d-1}{2}}{d-k-1}
\binom{n-\frac{d+3}{2}-j}{n-d-j}\Bigg) & \textrm{if $d$ is odd} 
\end{array} \right. 
\end{displaymath}
\end{proposition}
\vspace{3mm}

\begin{corollary}\label{c1}
For every positive integer $d$, there is an integer $N(d)$, such that for all $n\geq N(d)$ the face vector of a neighborly cubical $d$-polytope with $2^n$ vertices peaks at $\lfloor 2(d-1)/3 \rfloor$.
\end{corollary}
\begin{proof}
The proof is based on straightforward computations, we give only the main steps of the proof. Let $f_k(d,n)$ denote the $k^{th}$ component of the $f$-vector of a neighborly cubical $d$-polytope with $2^n$ vertices. Let $l=\lfloor 2(d-1)/3 \rfloor$ and $m=\lfloor (d-2)/2 \rfloor$. We have to show that $f_{l-1}(d,n)\leq f_l(d,n)$ and $f_{l}(d,n)\geq f_{l+1}(d,n)$ if $n$ is large enough. The latter inequality follows directly from Theorem 1 in \cite{maj1}.

First, let $d$ be even. According to the formula of Proposition \ref{pfk}, we have to show that for large $n$,
$$
\sum_{i=0}^{m} 
\textstyle
\Big(\binom{d-i-1}{l-i}-2\binom{d-i-1}{l-1-i}+\binom{i}{l-d+i+1}-2\binom{i}{l-d+i}\Big)\binom{n-d+i}{i}\geq 0$$
We deal with the first $m$ terms and the last term of the above sum separately. It can be shown that the following sum is less than the above one
$$
\textstyle\binom{n-d+m}{m}-2\cdot\displaystyle\sum_{i=0}^{m-1} 
\textstyle \binom{d-i-1}{l-1-i}\binom{n-d+i}{i}$$
Consequently, it is enough to prove that this sum is non-negative for large $n$. By using the identity $\binom{a+1}{b}=\binom{a}{b}+\binom{a-1}{b-1}+\cdots \binom{a-b}{0}$, we obtain that
$$
2\cdot\displaystyle\sum_{i=0}^{m-1} 
\textstyle \binom{d-i-1}{l-1-i}\binom{n-d+i}{i}\leq 2\cdot\binom{d}{l-1}\binom{n-d+m-1}{m-1}\leq \textstyle\binom{n-d+m}{m} $$
if $n\geq 2m\cdot\textstyle\binom{d}{l-1}+d-m$.

Now, let $d$ be  odd. To prove that $f_{l-1}(d,n)\leq f_l(d,n)$ we need to investigate the second summation of the formula in Proposition \ref{pfk} for odd $d$: $$a_k=2^{n-k}\displaystyle\sum_{j=0}^{n-d}2^{-j}\textstyle\binom{m+1}{d-k-1}
\binom{n-m-j}{n-d-j} \hspace{3mm} \textrm { for } \hspace{2mm} 0\leq k \leq d.$$
Due to the proof of the case  of even $d$, we have $f_{l-1}(d,n)-a_{l-1}\leq f_l(d,n)-a_{l}$, hence it remains to prove that $a_{l-1}\leq a_{l}$. Since $a_{l}- a_{l-1}\geq 0$ if
$$\Big(\textstyle\binom{m+1}{d-l-1}-2\binom{m+1}{d-l}\Big)\cdot\displaystyle\sum_{j=0}^{n-d}2^{-j}\textstyle
\binom{n-m-j}{n-d-j} \geq 0$$ 
it is enough to verify that $\textstyle\binom{m+1}{d-l-1}-2\binom{m+1}{d-l}\geq0$. Therefore, we obtain that the positivity of $a_{l}- a_{l-1}$ does not depend on $n$.
\end{proof}
\begin{remark}By estimating the $f$-vectors of cyclic polytopes (with Stirling's formula), Ziegler \cite{zi1} showed that the $f$-vector of a cyclic polytope with $n$ vertices approximately peaks at $\lfloor 3(d-1)/4 \rfloor$ if $n$ is large compared to $d$. Analogously to the above corollary, one can prove that for every positive integer $d$, there is a positive integer $N(d)$, such that for all $n\geq N(d)$ the face vector of a cyclic  $d$-polytope with $n$ vertices peaks at $\lfloor 3(d-1)/4 \rfloor$. According to computer experiments, Schmitt conjectured (Conjecture 2.4.3 in \cite{sch}) that for every positive integer $d$, $N(d)$ can be less than or equal to $d^2$. 
\end{remark} 
\begin{conjecture}
Let $d$ be a positive integer. For every $\lfloor d/3 \rfloor\leq k \leq \lfloor 2(d-1)/3 \rfloor$, there is a  neighborly cubical $d$-polytope, whose face vector peaks at $k$.
\end{conjecture}


\noindent  

\begin{thebibliography}{9}                                                                                            
\vspace{1.2mm} 

\bibitem {adi}  \textsc{R. M. Adin}:\ \textit{A new cubical $h$-vector},  Discrete Mathematics 157 (1996)\vspace{0.7mm} 3-14.

\bibitem {bil}  \textsc{E. K. Babson, L. J. Billera, C. S. Chan}:\ \textit{Neighborly cubical spheres and a cubical lower bound conjecture},  Israel Journal of Mathematics, 102, (1997),\vspace{0.7mm}  297-315.

\bibitem {bj1} \textsc{A. Bj\"orner}: \textit{Partial unimodality for f-vectors of simplicial polytopes and spheres}, Contemporary Mathematics\vspace{0.7mm} 178 (1994) 45-45.

\bibitem{B1} \textsc{F. Brenti}, \textit{Log-concave and unimodal sequences in algebra, combinatorics, and geometry: an update}, Contemporary Math., 178 (1994), pp. \vspace{0.7mm} 71-89.

\bibitem {eck}  \textsc{J. Eckhoff}: \textit{Combinatorial properties of $f$-vectors of convex polytopes}, Normat 54 (4)\vspace{0.7mm} (2006) 146-159.

\bibitem {jos} \textsc{M. Joswig, G. M. Ziegler}:\ \textit{Neighborly cubical polytopes}, Discrete  and Computational Geometry, Volume 24, Issue 2-3, (2000),\vspace{0.7mm} 325-344.

\bibitem {gru} \textsc{B. Gr\"unbaum}:\ \textit{Convex Polytopes}, Interscience, London, 1967; revised edition (V. Kaibel, V. Klee and G.M. Ziegler, editors), Graduate Texts \vspace{0.7mm}in Math., Springer-Verlag, (2003).
 
\bibitem {lee} \textsc{C. W. Lee}:\ \textit{Some recent results on convex polytopes}, in: "Mathematical Developments arising from Linear Programming"  (J. C. Lagarias and M. J. Todd, eds.), Contemporary Mathematics 114 \vspace{0.7mm}(1990),
 3-19. 

\bibitem {maj} \textsc{L. Major}:\ \textit{Unimodality and log-concavity of $f$-vectors for cyclic and ordinary polytopes}, Discrete Applied Mathematics,  161, 10-11,\vspace{0.7mm} (2013),  1669-1672.

\bibitem {maj1} \textsc{L. Major, Sz. T\'oth}:\ \textit{The Unimodality Conjecture for cubical polytopes}, preprint: arXiv:1501.00430v2 (2015).  

\bibitem {sch} \textsc{M. W. Schmitt}: \textit{On unimodality of f-vectors of convex polytopes}, Diplomarbeit, Technische
Universität Berlin, 2009. Betreuung: G. M. Ziegler

\bibitem {SZ1} \textsc{M. W. Schmitt, G. M. Ziegler},
\textit{Ten Problems in Geometry}, in {Shaping Space}, ed. {M. Senechal}, {Springer, New York}, {(2013)}, pp. \vspace{0.7mm}{279-289.}


\bibitem {sta} \textsc{R. P. Stanley}:\ \textit{The number of faces of simplicial polytopes and spheres}, Discrete Geometry and Convexity,  Ann. New York Acad. Sci. 440, edited by J.E. Goodman, et al.,\vspace{0.7mm}(1985), 212-223. 

\bibitem {sta1} \textsc{R. P. Stanley}: \textit{Log-concave and unimodal sequences in algebra, combinatorics, and geometry, in Graph theory and its applications} :  Ann. New York Acad. Sci., 576 (1), 
 \vspace{0.7mm}(1989) 500-535. 

\bibitem {wang}\textsc{X. T. Su, Y. Wang} On unimodality problems in Pascal's triangle, Electron. J. Combin. 15 (2008), Research Paper 113, 

\bibitem {zie} \textsc{G. M. Ziegler}:\ \textit{Lectures on Polytopes}, vol.
 152 of Graduate Texts in Mathematics, Springer-Verlag, New York, \vspace{0.7mm} 1995.

 \bibitem {zi1}  \textsc{G. M. Ziegler}: \textit{Convex polytopes: extremal constructions and $f$-vector shapes}, in: E. Miller, V. Reiner, B. Sturmfels (Eds.), Geometric Combinatorics,
in: IAS/Park City Mathematics Series, vol. 13, American Mathematical Society, Institute for Advanced Study, (2007) 617-691. arXiv:math.MG/
0411400v1, (2004)
 
\end{thebibliography}
\end{document}